\documentclass[11pt,reqno]{amsart}%
\usepackage{amsmath, amssymb, amsfonts}
\usepackage{amsthm}
\usepackage{version}
\usepackage{verbatim}%
\usepackage{amsmath}%
\usepackage{amsfonts}%
\usepackage{amssymb}%
\usepackage{graphicx}
\providecommand{\U}[1]{\protect\rule{.1in}{.1in}}
\theoremstyle{definition}
\newtheorem{theorem}{Theorem}

\newtheorem{definition}[theorem]{Definition}
\newtheorem{example}[theorem]{Example}

\newtheorem{lemma}[theorem]{Lemma}

\newtheorem{proposition}[theorem]{Proposition}
\newtheorem{remark}[theorem]{Remark}

\numberwithin{equation}{section}
\numberwithin{theorem}{section}

\begin{document}
\title[Obstacle problems on homogeneous groups]{The Obstacle Problem for Subelliptic Non-divergence\linebreak Form Operators
on Homogeneous Groups}
\author{Marie Frentz}
\address{Department of Mathematics, London School of Economics, London, United Kingdom}
\email{marie.frentz@gmail.com}
\author{Heather Griffin}
\address{Department of Mathematics, University of Arkansas, Fayetteville, USA}
\email{hgriffi@uark.edu}
\maketitle

\begin{abstract}
The main result established in this paper is the existence and uniqueness of
strong solutions to the obstacle problem for a class of subelliptic operators
in non-divergence form. The operators considered are structured on a set of
smooth vector fields in $\mathbf{R}^{n}$\noindent$,X=\{X_{0},X_{1}%
,...,X_{q}\}$, $q\leq n$, satisfying H{\"{o}}rmander's finite rank condition.
In this setting, $X_{0}$ is a lower order term while $\{X_{1},...,X_{q\}}$ are
building blocks of the subelliptic part of the operator. In order to prove
this, we establish an embedding theorem under the assumption that the set
$\{X_{0},X_{1},...,X_{q}\}$ generates a homogeneous Lie group. Furthermore, we
prove that any strong solution belongs to a suitable class of H\"{o}lder
continuous functions.

\vspace{.2in}



\end{abstract}

%\tableofcontents

\section{Introduction}

Obstacle problems form an important class of problems in analysis and applied
mathematics as they appear, in particular, in the mathematical study of
variational inequalities and free boundary problems. The classical obstacle
problem involving the Laplace operator is to find the equilibrium position of
an elastic membrane, whose boundary is held fixed, and which is constrained to
lie above a given obstacle. This problem is closely related to the study of
minimal surfaces and to inverse problems in potential theory. Other
applications where obstacle problems occur, involving the Laplace operator or
more general operators, include control theory and optimal stopping, financial
mathematics, fluid filtration in porous media, constrained heating and
elasto-plasticity. As classical references for obstacle problems and
variational inequalities, as well as their applications, we mention Frehse
\cite{Fr72}, Kinderlehrer-Stampacchia \cite{KS80}, \cite{KS} and Friedman
\cite{FA82}. For an outline of the modern approach to the regularity theory of
the free boundary, in the context of the obstacle problem, we refer to
Caffarelli \cite{C98}.

In this paper we continue to develop a theory for the obstacle problem for a
general class of second order subelliptic partial differential equations in
non-divergence form modeled on a system of vector fields satisfying
H{\"{o}}rmander's finite rank condition. In particular, we consider operators
\begin{equation}
\mathcal{H}=\sum_{i,j=1}^{q}a_{ij}(x)X_{i}X_{j}+\sum_{i=1}^{q}b_{i}%
(x)X_{i}-a_{0}(x)X_{0},\ \ x\in\mathbf{R}^{n},\ n\geq3, \label{operator}%
\end{equation}
where $q\leq n$ is a positive integer. In Section \ref{assumptions}, we will
state the assumptions in detail. To formulate the obstacle problem, let
$\mathcal{H}$ be as in (\ref{operator}), and let $f,g,\varphi,\gamma
:\overline{\Omega}\rightarrow\mathbf{R}^{n}$ be continuous and bounded
functions such that $g\geq\varphi$ on $\overline{\Omega}$. We consider the
problem
\begin{equation}%
\begin{cases}
\max\{\mathcal{H}u(x)+\gamma(x)u(x)-f(x),\varphi(x)-u(x)\}=0, & \text{in}%
\ \Omega,\\
u(x)=g(x), & \text{on}\ \partial\Omega.
\end{cases}
\label{e-obs}%
\end{equation}
We say that $u$ is a strong solution to (\ref{e-obs}) if $u\in\mathcal{S}%
_{X,loc}^{1}(\Omega)\cap C(\overline{\Omega})$ satisfy the differential
inequality (\ref{e-obs}) almost everywhere in $\Omega$, while the boundary
datum is attained at all points of $\partial\Omega.$ Here $\mathcal{S}%
_{X,loc}^{1}$ denotes certain intrinsic Sobolev spaces, defined in Subsection
\ref{fspace}. The main result is the following.

\begin{theorem}
\label{obstacle} Under the assumptions in Subsection \ref{assumptions}, there
exists a unique strong solution to the obstacle problem in \eqref{e-obs}.
Furthermore, given $p$, $1\leq p<\infty$, and an open subset $\Omega^{\prime
}\subset\subset\Omega$ there exists a positive constant $c$, depending on
$\mathcal{H}$, $\Omega^{\prime}$, $\Omega$, $p$, $||f||_{L^{\infty}(\Omega)}$,
$||\gamma||_{L^{\infty}(\Omega)}$, $||g||_{L^{\infty}(\Omega)}$ and
$||\varphi||_{L^{\infty}(\Omega)}$, such that
\begin{equation}
||u||_{\mathcal{S}_{X}^{p}(U)}\leq c. \label{Sp bound}%
\end{equation}

\end{theorem}

To briefly put Theorem \ref{obstacle} into context we first consider the
parabolic case, that is when $q=n-1$ and $X=\{X_{0},X_{1},...,X_{q}\}$ is
identical to $\{\partial_{t},\partial_{x_{1}},...,\partial_{x_{n}}\}$. Then
there is an extensive literature on the existence of generalized solutions to
the obstacle problem in \eqref{e-obs} in Sobolev spaces, starting with the
pioneering papers \cite{FA75}, \cite{vM}, \cite{vM1} and \cite{McK}. The most
extensive and complete treatment of the obstacle problem for the heat equation
is due to Caffarelli, Petrosyan and Shahgholian \cite{CPS} and we refer to
\cite{CPS} for further references.

When $q<n,$ that is in the subelliptic case, with a general lower order term
$X_{0}$ there are, to our knowledge, no results concerning the problem in
\eqref{e-obs}. However, \cite{DGP07} and \cite{DGS03} treat the subelliptic
case when $X_{0}$ is lacking. Existence, uniqueness and regularity results for
solutions in the special case when $X_{0}=\partial_{t}$ are contained
in\ \cite{F11} and \cite{FGN12}. Similar results, but in the case of second
order differential operators of Kolmogorov type, are contained in \cite{FPP},
\cite{FNPP} and \cite{P08}. We note that the results presented here are new
due to the presence of the general lower order term $X_{0},$ and neither of
the above mentioned results cover the class of operators studied here, as
demonstrated in Section \ref{exempel}.

The proof of Theorem \ref{obstacle} is based on the classical penalization
technique introduced by Lewy and Stampacchia in \cite{LS69}. In particular, we
consider a family $(\beta_{\varepsilon})_{\varepsilon\in(0,1)}$ of smooth
functions such that, for fixed $\varepsilon\in(0,1),$ $\beta_{\varepsilon}$ is
an increasing function,%
\begin{equation}
\beta_{\varepsilon}(0)=0,\ \ \ \beta_{\varepsilon}(s)\leq\varepsilon
,\mbox{ whenever }s>0, \label{beta1}%
\end{equation}
and such that
\begin{equation}
\lim_{\varepsilon\rightarrow0}\beta_{\varepsilon}(s)=-\infty,\mbox{
whenever }s<0. \label{beta2}%
\end{equation}
A key step in the proof of Theorem \ref{obstacle} is to consider the penalized
problem%
\begin{equation}
\left\{
\begin{array}
[c]{ll}%
\mathcal{H}^{\delta}u_{\varepsilon,\delta}+\gamma^{\delta}u_{\varepsilon
,\delta}=f^{\delta}+\beta_{\varepsilon}(u_{\varepsilon,\delta}-\varphi
^{\delta})\ \ \ \ \  & \text{in }\Omega,\\
u_{\varepsilon,\delta}=g^{\delta} & \text{on }\partial\Omega,
\end{array}
\right.  \label{penalized2}%
\end{equation}
where the superscript $\delta$, $\delta\in(0,1)$, indicate certain mollified
versions of the objects at hand. The subscripts $\varepsilon,\delta$ in
$u_{\varepsilon,\delta}$ indicate that the solution depends on $\varepsilon$
through the penalizing function $\beta_{\varepsilon}$ and on $\delta$ through
the mollifier. We first prove that a classical solution to the problem in
(\ref{penalized2}) exists. By a classical solution we mean that
$u_{\varepsilon,\delta}\in C_{X}^{2,\alpha}(\Omega)\cap C(\overline{\Omega}),$
where the H{\"{o}}lder space $C_{X}^{2,\alpha}(\Omega)$ is defined in terms of
the intrinsic distance induced by the vector fields. In particular, this imply
that (\ref{penalized2}) is in fact satisfied pointwise.

Thereafter, a monotone iterative method is used to prove that $u_{\varepsilon
,\delta}$ is the limit of a sequence $\{u_{\varepsilon,\delta}^{j}%
\}_{j=1}^{\infty}$ where $u_{\varepsilon,\delta}^{j}\in C_{X}^{2,\alpha
}(\Omega)\cap C(\overline{\Omega})$. A key step in the argument is to ensure
compactness in $C_{X,loc}^{2,\alpha}(\Omega)\cap C(\overline{\Omega})$ of the
sequence constructed, which requires the use of certain a priori estimates. In
particular, we use interior Schauder estimates to conclude that there exists a
solution $u_{\varepsilon,\delta}$ to the problem in (\ref{penalized2}) such
that $u_{\varepsilon,\delta}\in C_{X,loc}^{2,\alpha}(\Omega)\cap
C(\overline{\Omega})$.

The final step is then to consider limits as $\varepsilon$ and $\delta$ tend
to $0$ and to prove that $u_{\varepsilon,\delta}\rightarrow u$ where $u$ is a
strong solution to the obstacle problem in \eqref{e-obs}. However, the
penalization technique only allows us to establish quite weak bounds on
$u_{\varepsilon,\delta}$ given that those bounds should be independent of
$\varepsilon$ and $\delta$. Hence, to prove that as $\varepsilon$,
$\delta\searrow0$ the function $u_{\varepsilon,\delta}\rightarrow u$ weakly in
$\mathcal{S}_{X,loc}^{p}$, $p\in\lbrack1,\infty)$, we use a priori interior
$\mathcal{S}_{X}^{p}$ estimates. To be able to subsequently conclude that in
fact $u_{\epsilon,\delta}\rightarrow u$ in $C_{X,loc}^{1,\alpha}(\Omega)\cap
C(\overline{\Omega})$, we also prove the following theorem.

\begin{theorem}
\label{Embedding}Under the assumptions in Subsection \ref{assumptions}, let
$\Omega^{\prime}\subset\subset\Omega$. If $u\in\mathcal{S}^{p}(\Omega)$, for
some $p\in(Q/2,Q)$, then%

\begin{equation}
||u||_{C^{1,\alpha}(\Omega^{\prime})}\leq c||u||_{\mathcal{S}^{p}(\Omega)}
\label{badda}%
\end{equation}
for$\ \alpha=(p-Q)/p$. Moreover, the constant $c$ only depend on $\mathcal{G}%
$, $\mu$, $p$, $s$, $\Omega$ and $\Omega^{\prime}.$
\end{theorem}

In the context of the circle of techniques and ideas used in this paper it is
also fair to mention \cite{BB00b}, \cite{BZ11}, \cite{FGN12} and \cite{FPP}.

Throughout the paper, when we write that a constant $c$ depends on the
operator $\mathcal{H}$, $c=c(\mathcal{H})$, we mean that the constant $c$
depends on $n$, $q$, $X=\{X_{0},X_{1},...,X_{q}\}$, $\{a_{ij}\}_{i,j=1}^{q}$,
$\{b_{i}\}_{i=1}^{q}$ and$\ \lambda$. Furthermore, if $\alpha$ and $\Omega$
are given, then $c$ only depends on $||a_{ij}||_{C_{X}^{0,\alpha}(\Omega)}$,
$||b_{i}||_{C_{X}^{0,\alpha}(\Omega)}$, and not on any other properties of
these coefficients.

The remainder of this paper is organized as follows. Subsection
\ref{assumptions} contains assumptions on the vector fields, the operator, and
the domain for which our results hold. In Section \ref{prel}, which is of
preliminary nature, we introduce some notation as well as some basic facts
about homogeneous groups and subelliptic metrics, in particular, we account
for the proper function spaces. In Section \ref{estm} we present some
estimates for subelliptic operators. Section \ref{SecProof} is devoted to the
proof of the main theorem, and in Section \ref{SecEm} we prove the embedding
theorem. Finally, in Section \ref{exempel}, we give some examples of operators
to which our results apply. In particular, we demonstrate when and how our
results overlap with known ones and provide the reader with examples for which
our results are new, and not previously considered in the literature.

\section{Assumptions\label{assumptions}}

Here we present the assumptions made to be able to prove Theorem
\ref{obstacle}.

\textbf{The vector fields.} $X=\{X_{0},X_{1},...,X_{q}\}$ is a system of
smooth vector fields in $\mathbf{R}^{n}$ satisfying two main conditions. The
first of which is H\"{o}rmander's finite rank condition. To further explain
this, recall that the Lie-bracket between two vector fields $X_{i}$ and
$X_{j}$ is defined as $[X]_{i,j}=[X_{i},X_{j}]=X_{i}X_{j}-X_{j}X_{i}.$ For an
arbitrary multiindex $\alpha=(\alpha_{1},..,\alpha_{\ell})$, $\alpha_{k}%
\in\{0,1,..,q\}$, we define weights
\[
w_{0}=2\text{ and }w_{i}=1\text{ for }i=1,...,q.
\]
Using this we set%
\begin{equation}
|\alpha|=\sum_{i=1}^{\ell}w_{\alpha_{i}} \label{multiindex}%
\end{equation}
and define the commutator $[X]_{\alpha}$ of length $|\alpha|$ by
\[
\lbrack X]_{\alpha}^{{}}=[X_{\alpha_{\ell}},[X_{\alpha_{\ell-1}}%
,...[X_{\alpha_{2}},X_{\alpha_{1}}]]].
\]
$X=\{X_{0},X_{1},...,X_{q}\}$ is said to satisfy H{\"{o}}rmander's finite rank
condition, introduced in \cite{H67}, if there exists an integer $s$,
$s<\infty$, such that
\begin{equation}
\text{Lie}(X_{0},X_{1},\ldots,X_{q})=\{[X]_{\alpha}^{{}}:\alpha_{i}%
\in\{1,...,q\},\ |\alpha|\leq s\}\text{ spans }\mathbf{R}^{n}\text{ at every
point}. \label{1.1x}%
\end{equation}

Moreover, we assume that there exists a family of dilations , $\{D_{\lambda
}\}_{\lambda>0}$ in $\mathbf{R}^{n},$ such that $X_{1}$, $X_{2}$,
$\ldots,X_{q}$ are of $D_{\lambda}$-homogeneous degree one, and $X_{0}$ is of
$D_{\lambda}$-homogeneous degree 2.

These two conditions are enough to ensure the existence of a composition law
$\circ$ such that the triplet $(\mathbf{R}^{n},\circ,D_{\lambda})$ is a
homogeneous Lie group where the $X_{i}$'s are left invariant, see Subsection
\ref{homogen}. However, this homogeneity assumption is \textbf{only} essential
to proving the embedding theorem, Theorem \ref{Embedding}. This means that,
should the embedding theorem become available for general H\"{o}rmander vector
fields, then this proof carries over directly to this more general case.

\textbf{The coefficients. }Concerning the $q\times q$ matrix-valued function
$A=A(x)=\{a_{ij}(x)\}=\{a_{ij}\}$ and $a_{0}$ we assume that $A=\{a_{ij}\}$ is
real symmetric, with bounded and measurable entries and that there exists
$\lambda\in\lbrack1,\infty)$ such that
\begin{equation}
\lambda^{-1}|\xi|^{2}\leq\sum_{i,j=1}^{q}a_{ij}(x)\xi_{i}\xi_{j}\leq
\lambda|\xi|^{2},\ \ \lambda^{-1}\leq a_{0}(x)\leq\lambda,\ \ \ \mbox{
whenever }x,\in\mathbf{R}^{n},\ \xi\in\mathbf{R}^{q}. \label{1.2}%
\end{equation}
Concerning the regularity of $a_{ij}$ and $b_{i}$ we will assume that $a_{ij}
$ and $b_{i}$ have further regularity beyond being only bounded and
measurable. In fact, we assume that
\begin{equation}
a_{ij},\ b_{i}\in C_{X,loc}^{0,\alpha}(\mathbf{R}^{n})\mbox{ whenever }i,j\in
\{1,..,q\}, \label{1.2+}%
\end{equation}
for $\alpha\in(0,1),$ where $C_{X,loc}^{0,\alpha}(\mathbf{R}^{n})$ is the
space of functions which are bounded and H{\"{o}}lder continuous on every
compact subset of $\mathbf{R}^{n}$. Here the subscript $X$ indicates that
H{\"{o}}lder continuity is defined in terms of the Carnot-Carath\'{e}odory
distance induced by the set of vector fields $X$, see Section \ref{fspace}. In
particular, by (\ref{1.2}) we may divide the entire equation by $a_{0},$ and
hence consider (\ref{operator}) with $a_{0}=1$.

\textbf{The domain. }$\Omega$ is assumed to be a bounded domain such that
there exists, for all $\varsigma\in\partial\Omega$ and in sense of Definition
\ref{eter}, an exterior normal $v$ to $\overline{\Omega}$ relative
$\widetilde{\Omega},$ such that $C(\varsigma)v\neq0.$ Here $\widetilde{\Omega
}$ is a neighborhood of $\overline{\Omega}$ and $C(\cdot) $ is the matrix
valued function given by $(X_{1},...,X_{q})^{T}=C(x)\cdot(\partial_{x_{1}%
},...,\partial_{x_{n}})^{T}$. The assumption $C(\varsigma)v\neq0$ assures that
(\ref{fafa}) holds, and thus, that we can use Theorem \ref{Bony}.

\textbf{The equation.} Let $f,\gamma,g,\varphi:\overline{\Omega}%
\rightarrow\mathbf{R}^{n}$ be such that $g\geq\varphi$ on $\overline{\Omega}$
and assume that $f,\gamma,g,\varphi$ are continuous and bounded on
$\overline{\Omega}$, with $\gamma\leq\gamma_{0}<0$.

Concerning the obstacle $\varphi$ we assume that $\varphi$ is Lipschitz
continuous on $\overline{\Omega}$, where Lipschitz continuity is defined in
terms of the intrinsic homogeneous distance. We also assume that there exists
a constant $c\in\mathbf{R}^{+}$ such that
\begin{equation}
\sum_{i,j=1}^{q}\zeta_{i}\zeta_{j}\int_{\Omega}X_{i}X_{j}\psi(x)\varphi
(x)dx\geq c|\zeta|^{2}\int_{\Omega}\psi(x)dx \label{dobs}%
\end{equation}
for all $\zeta\in\mathbf{R}^{q}$ and for all nonegative test functions
$\psi\in C_{0}^{\infty}(\Omega)$. The reader might want to think of this as a
convexity assumption.

\section{Preliminaries\label{prel}}

In this section we introduce notations and concepts to be used throughout the
paper. For a more detailed exposition we refer to the monograph \cite{BLU}
written by Bonfiglioli, Lanconelli and Uguzzoni.

In the following we assume that $X=\{X_{0},X_{1},...,X_{q}\}$ satisfies
\eqref{1.1x}. From now on we will write $Xf$ when a vector field $X$ act on a
function $f$ as a differential operator. We begin by defining the
Carnot-Carath\'{e}odory distance, also known as the control distance, see
\cite{BBP} and \cite{NSW85}.

\begin{definition}
For any $\delta>0,$ let $\Gamma(\delta)$ be the class of all absolutely
continuous mappings $\gamma:[0,1]\rightarrow\Omega$ such that%
\[
\gamma^{\prime}(t)=\sum_{i=0}^{q}\lambda_{i}(t)X_{i}(\gamma(t))\ \ \text{a.e.
}t\in(0,1)
\]
with $|\lambda_{0}(t)|\leq\delta^{2}$ and $|\lambda_{i}(t)|\leq\delta$ for
$i=1,...,q$. Then we define the Carnot-Carath\'{e}odory distance between two
points $x,y\in\Omega$ to be%
\[
d(x,y)=\inf\{\delta:\exists\gamma\in\Gamma(\delta)\text{ with }\gamma
(0)=x\text{ and }\gamma(1)=y\}.
\]

\end{definition}

It is a non-trivial result that any two points in $\Omega$ can be connected by
such a curve, and the proof relies on a connectivity result of Chow,
\cite{Cho40}. We remark that the Carnot-Carath\'{e}odory distance $d$ is in
fact a quasi-distance because the triangle inequality does not hold. Instead,
the inequality has the form
\[
d(x,z)\leq C(d(x,y)+d(y,z))
\]
where the constant $C$ depends on the vector fields. Moreover, there exist
constants $c_{1},c_{2}$, depending on $\Omega$, such that
\begin{equation}
c_{1}|x-y|\leq d(x,y)\leq c_{2}|x-y|^{1/s}\ \ \text{for all }x,y\in\Omega,
\label{Carnot-Euclid}%
\end{equation}
where $s$ is the rank in the H\"{o}rmander condition, see Proposition 1.1 in
\cite{NSW85}. This is not immediate, but follows from \cite[Section 5]{BBP}.

\subsection{Homogeneous groups\label{homogen}}

Let $\circ$ be a given group law on $\mathbf{R}^{n}$ and suppose that the map
$(x,y)\mapsto y^{-1}\circ x$ is smooth. Then $\mathbf{G}=(\mathbf{R}^{n}%
,\circ)$ is called a \emph{Lie group}. $\mathbf{G}$ is said \emph{homogeneous}
if there exists a family of \emph{dilations} $\left(  D_{\lambda}\right)
_{\lambda>0}$ on $\mathbf{G,}$ which are also automorphisms, of the form
\begin{equation}
D_{\lambda}(x)=D_{\lambda}(x^{(1)},...,x^{(l)})=(\lambda x^{(1)}%
,...,\lambda^{l}x^{(l)})=(\lambda^{\sigma_{1}}x_{1},...,\lambda^{\sigma_{n}%
}x_{n}), \label{str1}%
\end{equation}
where $1\leq\sigma_{1}\leq...\leq\sigma_{n}.$ Note that in \eqref{str1} we
have that $x^{(i)}\in\mathbf{R}^{n_{i}}$ for $i\in\{1,...,l\}$ and
$n_{1}+....+n_{l}=n$. On $\mathbf{G}$ we define a homogeneous norm $||\cdot||
$ as follows; for $x\in\mathbf{R}^{n}$, $x\neq0,$ set%
\[
||x||=\rho\text{ \ \ if and only if \ \ }|D_{1/\rho}(x)|=1,
\]
where $|\cdot|$ denotes the standard Euclidean norm, and set $||0||=0.$ This
norm satisfies the following:

\begin{description}
\item[i)] $||D_{\lambda}(x)||=\lambda||x||$ for all $x\in\mathbf{R}%
^{n},\ \lambda>0.$

\item[ii)] The set $\{x\in\mathbf{R}^{n}:||x||=1\}$ coincides with the
Euclidean unit sphere.

\item[iii)] There exist $c(\mathbf{G})\geq1$ such that for every
$x,y\in\mathbf{R}^{n}$%
\[
||x\circ y||\leq c(||x||+||y||)\text{ \ \ and \ \ }||x^{-1}||\leq c||x||.
\]

\end{description}

\noindent We also define a quasidistance $d$ on $\mathbf{R}^{n}$ through%
\[
d(x,y)=||y^{-1}\circ x||.
\]
For this quasidistance there exist $c=c(\mathbf{G})$ such that for all
$x,y,z\in\mathbf{R}^{n}$ the following holds;

\begin{description}
\item[iv)] $d(x,y)\geq0$ and $d(x,y)=0$ if and only if $x=y$.

\item[v)] $c^{-1}d(y,x)\leq d(x,y)\leq cd(y,x).$

\item[vi)] $d(x,y)\leq c(d(x,z)+d(z,y)).$
\end{description}

The previously mentioned Carnot-Caratheodory distance is one example of an
appropriate distance function. Alternatively, one could begin by defining
$||x||=\sum_{j=1}^{n}|x_{j}|^{1/\sigma_{j}}$, with the induced distance
$d(x,y)=||x^{-1}\circ y||$ satisfying the properties above as well.

\noindent We define balls with respect to $d$ by%
\[
B(x,r)=\{y\in\mathbf{R}^{n}:d(x,y)<r\}.
\]
In particular, we note that $D_{r}(B(0,1))=B(0,r).$ Moreover, in \cite[p.
619]{S} it is proved that the Lebesgue measure in $\mathbf{R}^{n}$ is the Haar
measure of $\mathbf{G}$ and that%
\begin{equation}
|B(x,r)|=|B(0,1)|r^{Q}, \label{bollar}%
\end{equation}
where $Q$ is the natural number
\begin{equation}
Q:=n_{1}+2n_{2}+...+ln_{l}, \label{e-Q}%
\end{equation}
also called the \emph{homogeneous dimension} of $\mathbf{G}$.

The {convolution} of two functions $f,g,$ defined on $\mathbf{G,}$ is defined
as%
\[
(f\ast g)(\zeta)=\int_{\mathbf{R}^{N}}f(\zeta\circ\xi^{-1})g(\xi)d\xi
\]
whenever the integral is well defined. Let $P$ be a differential operator and
let $\tau_{\xi}$ be the {left translation operator}, i.e., $(\tau_{\xi
}f)(\zeta)=f(\xi\circ\zeta)$ whenever $f$ is a function on $\mathbf{G}$. A
differential operator $P$ is said to be {left invariant} if
\[
P(\tau_{\xi}f)=\tau_{\xi}(Pf).
\]
Further, we say that the differential operator $P$ is {homogeneous of degree
}$\delta$ if, for every test function $f,\ \lambda>0$ and $\xi\in
\mathbf{R}^{N}$,
\[
P(f(D(\lambda)\xi))=\lambda^{\delta}(Pf)(D(\lambda)\xi).
\]
Similarly, a function $f$ is {homogeneous of degree }$\delta$ if
\[
f(D((\lambda)\xi))=\lambda^{\delta}f(\xi)\mbox{ whenever
$\lambda >0$, $\ \xi \in \mathbf{R}^{n}$}.
\]
Note that if $P$ is a differential operator homogeneous of degree $\delta_{1}$
and if $f$ is a function homogeneous of degree $\delta_{2}$ then $fP$ is a
differential operator homogeneous of degree $\delta_{1}-\delta_{2}$ and $Pf$
is a function homogeneous of degree $\delta_{2}-\delta_{1}.$We conclude this
section with a proposition which will be used to prove the embedding theorem,
see \cite[Proposition 1.15]{F75}.

\begin{proposition}
\label{MMM}Let $f\in C^{1}(\mathbf{R}^{n}\backslash\{0\})$ be homogeneous of
degree $\delta.$ Then there exist $c=c(\mathcal{G},f)>0$ and $M=M(\mathcal{G}%
)>1$ such that%
\[
|f(x\circ y)-f(x)|+|f(y\circ x)-f(x)|\leq c||y||\cdot||x||^{\delta-1},
\]
for every $x,y$ such that $||x||\geq M||y||$.
\end{proposition}

\begin{comment}We shall also use the fact that there exists a constant $c=c(\mathbf{G})$ such
that%
\begin{equation}
c^{-1}d_{h}(x,y)\leq d(x,y)\leq cd_{h}(x,y),\label{equivNorm}%
\end{equation}
following \cite[Subsections 1.3, 5.1-5.2]{BLU} and the fact that $d$ is
homogeneous. \textbf{I HAVE CHECKED THIS USING THE DEFINITION, PROPOSITION
1.2.3 AND PROPOSITION 5.1.4 IN BLU. I GUESS THIS IS KNOWN, AND THEREFORE DID
NOT WANT TO WRITE IT UP, BUT I DO NOT KNOW OF A REFERENCE. SHOULD WE WRITE
MORE ABOUT THIS?}
\end{comment}

\subsection{Function spaces\label{fspace}}

Let $U\subset\mathbf{R}^{n}$ be a bounded domain and let $\alpha\in(0,1]$.
Given $U$ and $\alpha$ we define the {H\"{o}lder space $C_{X}^{0,\alpha}(U)$
as $C_{X}^{0,\alpha}(U)=\{u:U\rightarrow\mathbf{R}:\ ||u||_{C^{0,\alpha}%
(U)}<\infty\}$, where
\begin{align*}
||u||_{C_{X}^{0,\alpha}(U)}  &  =|u|_{C_{X}^{0,\alpha}(U)}+||u||_{L^{\infty
}(U)},\\
|u|_{C_{X}^{0,\alpha}(U)}  &  =\sup\left\{  \frac{|u(x,t)-u(y,t)|}%
{d(x,y)^{\alpha}}:x,y\in U\text{ and}\ x\neq y\right\}  .
\end{align*}
Given a multiindex $I=(i_{1},i_{2},...,i_{m})$, with $0\leq i_{j}\leq q$,
}${1}\leq j\leq m,$ {we define the weighted length of the multiindex, $|I|,$
as in (\ref{multiindex})$\ $and we set $X^{I}u=X_{i_{1}}X_{i_{2}}\cdots
X_{i_{m}}u$. Now, given a domain $U$, an exponent $\alpha$ and an arbitrary
non-negative integer $k$ we define $C_{X}^{k,\alpha}(U):=\{u:U\rightarrow
\mathbf{R}:\ ||u||_{C_{X}^{k,\alpha}(U)}<\infty\}$, where
\[
||u||_{C_{X}^{k,\alpha}(U)}=\sum_{|I|\leq k}||X^{I}u||_{C_{X}^{0,\alpha}(U)}.
\]
Sobolev spaces are defined as%
\[
\mathcal{S}_{X}^{p}(U)=\left\{  u\in L^{p}(U):X_{0}u,\ X_{i}u,\ X_{i}X_{j}u\in
L^{p}(U)\text{ for}\ i,j=1,...,q\right\}
\]
and we define the Sobolev norm of a function }$u$ by{%
\[
||u||_{\mathcal{S}_{X}^{p}(U)}=||u||_{L^{p}(U)}+\sum\limits_{i=0}^{q}%
||X_{i}u||_{L^{p}(U)}+\sum\limits_{i,j=1}^{q}||X_{i}X_{j}u||_{L^{p}(U)}.
\]
}Above the $L^{p}$-norms are taken with respect to the standard Euclidean
metric, in particular, we integrate with respect to the Lebesgue measure. {Let
$U\subset\mathbf{R}^{n}$ be a domain, not necessarily bounded. If $u\in
C_{X}^{k,\alpha}(V)$ for every compact subset $V$ of $U$, then we say that
$u\in C_{X,loc}^{k,\alpha}(U).$ Similarly, if $u\in\mathcal{S}_{X}^{p}(V)$ for
every compact subset $V$ of $U$, then we say that $u\in\mathcal{S}_{X,loc}%
^{p}(U).$ }

An important result about compactly supported test functions multiplied by
Sobolev functions is the following lemma \cite[Corollary 1]{BB00b}.

\begin{lemma}
\label{cutoff}If $u\in\mathcal{S}^{p}(\Omega),\ 1\leq p<\infty,$ and $\phi\in
C_{0}^{\infty}(\Omega),$ then $u\phi\in\mathcal{S}_{0}^{p}(\Omega).$
\end{lemma}

This lemma will be used when $\phi$ is a cutoff function. The existence of
smooth cutoff functions is not immediate, but by \cite[Lemma 5]{BB00b}, we
have the following.

\begin{lemma}
\label{cutoff2}For any $\sigma\in(0,1),r>0,k\in\mathbf{Z}_{+},$ there exists
$\phi\in C_{0}^{\infty}(\mathbf{R}^{n})$ with the following properties:%
\[
B_{\sigma r}\prec\phi\prec B_{\sigma^{\prime}r}\ \ \ \text{with }%
\sigma^{\prime}=(1+\sigma)/2;
\]%
\[
|X^{\alpha}\phi|\leq\frac{c(\mathcal{G},j)}{\sigma^{j-1}(1-\sigma)^{j}r^{j}%
}\text{ \ \ for all multiindices }|\alpha|=j\in\{1,...,k\}.
\]

\end{lemma}

\section{Estimates for subelliptic operators with drift\label{estm}}

Here we collect a number of theorems which concern subelliptic operators with
drift, all of which are important tools in the proof of the obstacle problem.
We begin with a result of Bony \cite[Theoreme 5.2]{B69} which is both a
comparison principle and a result on solvability of the Dirichlet problem.
Before we state the theorem we introduce the notion of an exterior normal.

\begin{definition}
\label{eter} A vector $v$ in $\mathbf{R}^{n}$ is an exterior normal to a
closed set $S\subset\mathbf{R}^{n}$ relative an open set $U$ at a point
$x_{0}$ if there exists an open standard Euclidean ball $B_{E}$ in
$U\backslash S$ centered at $x_{1}$ such that $x_{0}\in\overline{B_{E}}$ and
$v=\lambda(x_{1}-x_{0})$ for some $\lambda>0.$
\end{definition}

\begin{theorem}
\label{Bony}(Bony) Let $U\subset\mathbf{R}^{n}$ be a bounded domain and let
$H:=\sum_{i=1}^{r}Y_{i}^{2}+Y_{0}+\gamma=$ $\sum_{i,j=1}^{n}a_{ij}^{\ast
}\partial_{x_{i}x_{j}}+\sum_{i=1}^{n}a_{i}^{\ast}\partial_{x_{i}}+\gamma$.
Assume that the set of vector fields $Y=\{Y_{0},Y_{1},...,Y_{r}\}$ satisfies
H\"{o}rmander%
%TCIMACRO{\U{b4}}%
%BeginExpansion
\'{}%
%EndExpansion
s finite rank condition, that $\gamma(x)\leq\gamma_{0}<0$ for all $x\in U$ and
that $a_{ij}^{\ast},\ a_{i}^{\ast},\ \gamma\in C^{\infty}(U).$ In addition,
assume that for all $x\in U$ and for all $\xi\in\mathbf{R}^{n}$ the quadratic
form $\sum_{i,j=1}^{n}a_{ij}^{\ast}(x)\xi_{i}\xi_{j}\geq0.\ $Further, assume
that $D$ is a relatively compact subset of $U$ and that at every point
$x_{0}\in\partial D$ there exists an exterior normal $v$ such that%
\begin{equation}
\sum_{i,j=1}^{n}a_{ij}^{\ast}(x_{0})v_{i}v_{j}>0. \label{fafa}%
\end{equation}
Then, for all $g\in C(\partial D)$ and $f\in C(\overline{D})$, the Dirichlet
problem%
\[
\left\{
\begin{array}
[c]{cc}%
Hu=-f, & \mbox{in $D$},\\
u=g, & \mbox{on $\partial D$},
\end{array}
\right.
\]
has a unique solution $u\in C(\overline{D})$. Furthermore, if $f\in C^{\infty
}(D)$, then $u\in C^{\infty}(D)$ and if $f$ and $g$ are both positive then so
is $u$.
\end{theorem}

We remark that we cannot use this theorem directly since we only assume that
our coefficients $a_{ij}$ and $b_{i}$ are H\"{o}lder continuous. However, for
smooth coefficients and using linear algebra, our operator $\mathcal{H}$ in
(\ref{operator}) can be rewritten as a H\"{o}rmander operator in accordance
with Bony's assumptions. We will also use a Schauder type estimate, the
particular one we use can be found in \cite[Theorem 2.1]{BZ11}.

\begin{comment}\textbf{I'M NOT SURE THAT THE BOLD STATEMENT IS STILL TRUE WITH THE LOWER ORDER TERMS $\sum_{i=1}^q b_iX_i.$  HAVE YOU CHECKED THIS BEFORE?}\end{comment}

\begin{theorem}
\label{Schauder} (Schauder estimate) Assume that the operator $\mathcal{H}$ is
strucutured on a set of smooth H\"{o}rmander vector fields and that the
coefficients $a_{ij},b_{i}\in C_{X}^{0,\alpha}(\Omega)$ for some $\alpha
\in(0,1),\ a_{0}\in L^{\infty}(\Omega).$ Then for every domain $\Omega
^{\prime}\subset\subset\Omega$ there exists a constant $c,\ $depending on
$\Omega^{\prime}$, $\Omega$, $X$, $\alpha$, $\lambda$, $||a_{ij}%
||_{C_{X}^{0,\alpha}(\Omega)}$, $||b_{i}||_{C_{X}^{0,\alpha}(\Omega)}$ and
$||a_{0}||_{C_{X}^{0,\alpha}(\Omega)}$ such that for every $u\in
C_{X}^{2,\alpha}(\Omega)$ one has%
\[
||u||_{C_{X}^{2,\alpha}(\Omega^{\prime})}\leq c\left\{  ||\mathcal{H}%
u||_{C_{X}^{0,\alpha}(\Omega)}+||u||_{L^{\infty}(\Omega)}\right\}  .
\]

\end{theorem}

We emphasize that in \cite{BZ11} this is only proved when the lower order
terms $b_{i}\equiv0.$ However, by arguing as in the proof of Theorem 10.1 in
\cite{BB07} this also hold for $b_{i}\in C_{X}^{0,\alpha}(\Omega).$ This
Schauder estimate will be used together with an a priori $\mathcal{S}^{p}$
interior estimate to assure proper convergence of a constructed sequence,
converging to a solution to the obstacle problem. The proof is to be found in
\cite[Theorem 2.2]{BZ11}

\begin{theorem}
\label{a priori} (A priori $\mathcal{S}^{p}$ interior estimate) Assume that
the operator $\mathcal{H}$ is strucutured on a set of smooth H\"{o}rmander
vector fields and that the coefficients $a_{ij}\in C_{X}^{0,\alpha}$ for some
$\alpha\in(0,1).$ Then for every domain $\Omega^{\prime}\subset\subset\Omega$
there exists a constant $c,\ $depending on $\Omega^{\prime}$, $\Omega$, $X$,
$\alpha$, $\lambda$, $||a_{ij}||_{C_{X}^{0,\alpha}(\Omega)}$, $||b_{i}%
||_{C_{X}^{0,\alpha}(\Omega)}$ and $||a_{0}||_{C_{X}^{0,\alpha}(\Omega)}$ such
that for every $u\in\mathcal{S}_{X}^{p}(\Omega)$ one has%
\[
||u||_{\mathcal{S}_{X}^{p}(\Omega^{\prime})}\leq c\left\{  ||\mathcal{H}%
u||_{L^{p}(\Omega)}+||u||_{L^{\infty}(\Omega)}\right\}  .
\]

\end{theorem}

Also here, we can generalize the results in \cite{BZ11} to hold for $b_{i}\in
C_{X}^{0,\alpha}$, this time arguing as in Section 5.5 in \cite{FGN12}.

\section{Proof of Theorem \ref{obstacle}\label{SecProof}}

To prove Theorem \ref{obstacle} we will, as outlined in the introduction, use
the classical penalization technique and we let $(\beta_{\varepsilon
})_{\varepsilon\in(0,1)}$ be a family of smooth functions satisfying
\eqref{beta1} and \eqref{beta2}. For $\delta\in(0,1)$ we let $\mathcal{H}%
^{\delta}$ denote the operator obtained from $\mathcal{H}$ by regularization
of the coefficients $a_{ij},\ b_{i},\ i,j=1,...,q,$ using a smooth mollifier,
\[
\mathcal{H}^{\delta}=\sum_{i,j=1}^{q}a_{ij}^{\delta}(x)X_{i}X_{j}+\sum
_{i=1}^{q}b_{i}^{\delta}(x)X_{i}-X_{0},\ \ x\in%
%TCIMACRO{\U{211d} }%
%BeginExpansion
\mathbb{R}
%EndExpansion
^{n}.
\]
We also regularize $\varphi$, $\gamma$ and $f$ and denote the regularizations
$\varphi^{\delta}$, $\gamma^{\delta}$ and $f^{\delta}$ respectively.
Especially, we are able to extend these functions by continuity to a
neighborhood of $\Omega$. As stated in the introduction, see the discussion
above \eqref{dobs}, we assume that $\varphi$ is Lipschitz continuous on
$\overline{\Omega}$ and we denote its Lipschitz norm on $\overline{\Omega}%
\ $by $\mu.$ Then, since $g\geq\varphi$ on $\partial\Omega$ we see that
\[
g^{\delta}:=g+\mu\delta\geq\varphi^{\delta}\mbox{ on }\partial\Omega.
\]
Note that since $g$ is continuous, $g^{\delta}$ is also continuous and can
thus be used as boundary value function. As a first step we consider the
penalized problem%
\begin{equation}
\left\{
\begin{array}
[c]{ll}%
\mathcal{H}^{\delta}u+\gamma^{\delta}u=f^{\delta}+\beta_{\varepsilon
}(u-\varphi^{\delta})\ \ \ \ \  & \text{in }\Omega,\\
u=g^{\delta} & \text{on }\partial\Omega
\end{array}
\right.  \label{penalized}%
\end{equation}
and we prove that there exists a classical solution to this problem. This is
achieved in two steps, the first being:

\begin{theorem}
\label{thm2} Assume that $\mathcal{H}$ satisfies (\ref{1.1x}), (\ref{1.2}) and
(\ref{1.2+}), let $\Omega\subset\mathbf{R}^{n}$ be a bounded domain. Assume
that at every point $x_{0}\in\partial\Omega$ there exists an exterior normal
satisfying condition \eqref{fafa} in Theorem \ref{Bony}. Let $g\in
C(\partial\Omega)$ and let $h=h(x,u)$ be a smooth Lipschitz continuous
function, in the standard Euclidean sense, on $\overline{\Omega}. $ Then there
exists a classical solution $u\in C^{2,\alpha}(\Omega)\cap C(\overline{\Omega
})$ to the problem%
\[
\left\{
\begin{array}
[c]{ll}%
\mathcal{H}^{\delta}u=h(\cdot,u)\ \ \ \  & \text{in }\Omega,\\
u=g & \text{on }\partial\Omega.
\end{array}
\right.
\]
Furthermore, there exists a positive constant $c$, only depending on $h$ and
$\Omega$, such that%
\begin{equation}
\sup_{\Omega}|u|\leq c\left(  1+||g||_{L^{\infty}(\partial\Omega)}\right)  .
\label{hha}%
\end{equation}

\end{theorem}

\begin{proof}
To prove Theorem \ref{thm2} we will use the same technique as in the proof of
Theorem 3.2 in \cite{FPP}, i.e., a monotone iterative method. To start the
proof we note that, since $h=h(x,u)$ is a Lipschitz continuous function in the
standard Euclidean sense, there exists a constant $\mu$ such that
$|h(x,u)|\leq\mu(1+|u|)$ for $x\in\overline{\Omega}$. We let%
\begin{equation}
\label{init}u_{0}(x)=c(1+||g||_{L^{\infty}(\partial\Omega)})-1,
\end{equation}
for some constant $c$ to be chosen later, and we recursively define, for
$j=1,2,...,$%
\begin{equation}
\left\{
\begin{array}
[c]{ll}%
\mathcal{H}^{\delta}u_{j}-\mu u_{j}=h(\cdot,u_{j-1})-\mu u_{j-1}\ \ \ \ \  &
\text{in }\Omega,\\
u_{j}=g & \text{on }\partial\Omega.
\end{array}
\right. \label{monotone pde}%
\end{equation}
The linear Dirichlet problem in (\ref{monotone pde}) has been studied by Bony
in \cite{B69} and since the coefficients of the operator $\mathcal{H}^{\delta
}$ are smooth in a neighborhood of $\Omega$ it follows that $\mathcal{H}%
^{\delta}$ can be rewritten as a H\"{o}rmander operator in line with Theorem
\ref{Bony}. Hence, using Theorem \ref{Bony} we can conclude that a classical
solution $u_{j}\in C^{\infty}(\Omega)$ exists. In particular $u_{j}\in
C(\overline{\Omega})$ and combining Theorem \ref{Bony} with
(\ref{Carnot-Euclid})\ it follows that $u_{j}\in C_{loc}^{2,\alpha}(\Omega).$
We prove, by induction, that $\{u_{j}\}_{j=1}^{\infty}$ is a decreasing
sequence. By definition $u_{1}<u_{0}$ on $\partial\Omega$ and we can choose
the constant $c$ appearing in the definition of $u_{0}$, depending on $h$, so
that
\[
\mathcal{H}^{\delta}(u_{1}-u_{0})-\mu(u_{1}-u_{0})=h(\cdot,u_{0}%
)-\mathcal{H}^{\delta}u_{0}=h(\cdot,u_{0})+c(1+u_{0})\geq0
\]
holds. Thus, by the maximum principle, stated at the end of Theorem
\ref{Bony}, we conclude that $u_{1}<u_{0}$ on $\overline{\Omega}.$ Assume, for
fixed $j\in%
%TCIMACRO{\U{2115} }%
%BeginExpansion
\mathbb{N}
%EndExpansion
$, that $u_{j}<u_{j-1}.$ Then by the inductive hypothesis we see that
\begin{align*}
\mathcal{H}^{\delta}(u_{j+1}-u_{j})-\mu(u_{j+1}-u_{j})  & =h(\cdot
,u_{j})-h(\cdot,u_{j-1})-\mu(u_{j}-u_{j-1})\\
& =h(\cdot,u_{j})-h(\cdot,u_{j-1})+\mu|u_{j}-u_{j-1}|\geq0.
\end{align*}
Hence, by the maximum principle $u_{j+1}<u_{j}$ which proves that
$\{u_{j}\}_{j=1}^{\infty}$ is a decreasing sequence. By repeating this
calculation for $u_{j}+u_{0}$, we get the following bounds
\begin{equation}
-u_{0}\leq u_{j+1}\leq u_{j}\leq u_{0}.\label{uj-bound}%
\end{equation}
As $u_{j}\in C_{loc}^{2,\alpha}(\Omega)\cap C(\overline{\Omega})$ we can now
use Theorem \ref{Schauder} to conclude that
\begin{align}
||u_{j}||_{C^{2,\alpha}(U)}  & \leq c\left(  \sup_{\Omega}|u_{j}%
|+||\mathcal{H}^{\delta}u_{j}||_{C^{0,\alpha}(\Omega)}\right) \nonumber\\
& \leq c\left(  u_{0}+||h(\cdot,u_{j-1})||_{C^{0,\alpha}(\Omega)}+||\mu
(u_{j}-u_{j-1})||_{C^{0,\alpha}(\Omega)}\right) ,\label{uj-c2alfa}%
\end{align}
whenever $U$ is a compact subset of $\Omega$. Thus $||u_{j}||_{C^{2,\alpha
}(U)}$ is clearly bounded by some constant $c$ independent of $j$ due to
(\ref{uj-bound})-(\ref{uj-c2alfa}) and the fact that $h$ is Lipschitz. Thus
$\{u_{j}\}_{j=1}^{\infty}$ has a convergent subsequence in $C_{loc}^{2,\alpha
}(\Omega)$ and in the following we will denote the convergent subsequence
$\{u_{j}\}_{j=1}^{\infty}$. As $j\rightarrow\infty$ in (\ref{monotone pde}) we
have that%
\[
\left\{
\begin{array}
[c]{ll}%
\mathcal{H}^{\delta}u=h(\cdot,u)\ \ \ \ \  & \text{in }\Omega,\\
u=g & \text{on }\partial\Omega.
\end{array}
\right.
\]
We next prove that $u\in C(\overline{\Omega})$ by a barrier argument. For
fixed $\varsigma\in\partial\Omega$ and $\varepsilon>0,$ let $V$ be an open
neighborhood of $\varsigma$ such that%
\[
|g(x)-g(\varsigma)|\leq\varepsilon\mbox{ whenever }x\in V\cap\partial\Omega.
\]
Let $w:V\cap\overline{\Omega}\rightarrow%
%TCIMACRO{\U{211d} }%
%BeginExpansion
\mathbb{R}
%EndExpansion
$ be a function with the following properties:
\begin{align}
(i) & \mbox{$\mathcal{H}^{\delta }w\leq-1$ in $V\cap \Omega$},\nonumber\\
(ii) &
\mbox{$w>0$ in $V\cap \overline{\Omega}\backslash \{\varsigma\}$ and $w(\varsigma)=0.$}\nonumber
\end{align}
That such a function $w$ exists follows from the assumption that there exists
an exterior normal for all points on $\partial\Omega$ , see Definition
\ref{eter} and Remark \ref{remark} below. We define
\[
v^{\pm}(x)=g(\varsigma)\pm(\varepsilon+kw(x))\mbox{ whenever }x\in
V\cap\partial\Omega
\]
for some constant $k>0$ large enough to ensure that%
\[
\mathcal{H}^{\delta}(u_{j}-v^{+})\geq h(\cdot,u_{j-1})-\mu(u_{j-1}%
-u_{j})+k\geq0
\]
and that $u_{j}\leq v^{+}$ on $\partial(V\cap\Omega).$ Thus, the maximum
principle asserts that $u_{j}\leq v^{+}$ on $V\cap\Omega$ and likewise
$u_{j}\geq v^{-}$ on $V\cap\Omega.$ Note that $k$ can be chosen to depend on
the Lipschitz constant of $h$, $\mu$ and $u_{0}$ only and, in particular, $k$
can be chosen independent of $j.$ Passing to the limit we see that
\[
g(\varsigma)-\varepsilon-kw(x)\leq u(x)\leq g(\varsigma)+\varepsilon
+kw(x),\ \ \ \ x\in V\cap\Omega,
\]
and hence
\[
g(\varsigma)-\varepsilon\leq\underset{x\rightarrow\varsigma}{\lim\inf
}\ u(x)\leq\underset{x\rightarrow\varsigma}{\lim\sup}\ u(x)\leq g(\varsigma
)+\varepsilon
\]
where the limit $x\rightarrow\varsigma$ is taken through $x\in V\cap\Omega$.
Since $\varepsilon$ can be chosen arbitrarily we can conclude that $u\in
C(\overline{\Omega}).$ Finally, \eqref{hha} follows from an application of the
maximum principle.
\end{proof}

\begin{remark}
\label{remark} In the proof above we used barrier functions, plainly stating
that proper barrier functions exists. To see that this is actually the case,
let $\varsigma\in\partial\Omega$, then using our assumption on the domain
$\Omega$, see\ also Definition \ref{eter}, we see that there exists a standard
Euclidean ball in $\mathbf{R}^{n},$ $B_{E}(x_{0},\rho),$ with center $x_{0}%
\in\tilde{\Omega}\backslash\Omega$ and radius $\rho$, such that $B_{E}%
(x_{0},\rho)\subset\tilde{\Omega}$ and $\overline{B_{E}(x_{0},\rho)}%
\cap\overline{\Omega}=\{\varsigma\}.$ Using $x_{0}$ we define, for $K\gg1$,
\[
w(x)=e^{-K|\varsigma-x_{0}|^{2}}-e^{-K|x-x_{0}|^{2}}.
\]
Then, $w(\varsigma)=0$ and $w(x,t)>0$ for $x\in V\cap\overline{\Omega
}\backslash\{\varsigma\}.$ To see that $\mathcal{H}^{\delta}w\leq-1,$ we note
that since the coefficients of the operator $\mathcal{H}^{\delta}$ are smooth
in a neighborhood of $V\cap\overline{\Omega}$, $\mathcal{H}^{\delta}$ can be
rewritten as a H\"{o}rmander operator in line with Theorem \ref{Bony}. In
particular, using the notation of Theorem \ref{Bony} we have
\begin{align*}
\mathcal{H}^{\delta}w(x)  &  =-e^{-K|x-x_{0}|^{2}}\left(  4K^{2}\sum
_{i,j=1}^{n}a_{ij}^{\ast}(x)(x^{i}-x_{0}^{i})(x^{j}-x_{0}^{j})\right. \\
&  \ \ \ \ \ \ \ \ \ \ \ \ \ \ \ \ \ \ \ \ \ \ \ \ \ \ \ \ \left.
-2K\sum_{i=1}^{n}\left(  a_{ii}^{\ast}(x)+a_{i}^{\ast}(x)(x^{i}-x_{0}%
^{i})\right)  +\gamma(x)w(x)\right)  ,
\end{align*}
where $a_{ij}^{\ast},a_{i}^{\ast}$ and $\gamma$ denote the coefficients of the
H\"{o}rmander operator $\mathcal{H}^{\delta}$\ as stated in Theorem
\ref{Bony}. Hence, for $V$ small and choosing $K$ large enough, $\mathcal{H}%
^{\delta}w(x)\leq-1$ on $V\cap\overline{\Omega}.$ Thus, $w$ is indeed a proper
barrier function.
\end{remark}

\noindent\textbf{Proof of Theorem \ref{obstacle}.} We first note, using
Theorem \ref{thm2}, that the problem in (\ref{penalized}) has a classical
solution $u_{\varepsilon,\delta}\in C^{2,\alpha}(\Omega)\cap C(\overline
{\Omega})$. The assumption $\gamma<0$ enable us to use the maximum principle.
To proceed we first prove that
\begin{equation}
|\beta_{\varepsilon}(u_{\varepsilon,\delta}-\varphi^{\delta})|\leq c
\label{step1}%
\end{equation}
for some constant $c$ independent of $\varepsilon$ and $\delta.$ By definition
$\beta_{\varepsilon}\leq\varepsilon$ and hence we only need to prove the
estimate from below. Since $\beta_{\varepsilon}(u_{\varepsilon,\delta}%
-\varphi^{\delta})\in C(\overline{\Omega})$ this function achieves a minimum
at a point $(\varsigma,\tau)\in\overline{\Omega}.$ Assume that $\beta
_{\varepsilon}(u_{\varepsilon,\delta}(\varsigma)-\varphi^{\delta}%
(\varsigma,))\leq0,$ otherwise we are done. If $\varsigma\in\partial\Omega$,
then, since $g\geq\varphi$
\[
\beta_{\varepsilon}(u_{\varepsilon,\delta}(\varsigma)-\varphi^{\delta
}(\varsigma))=\beta_{\varepsilon}(g^{\delta}(\varsigma)-\varphi^{\delta
}(\varsigma))\geq0.
\]
On the other hand, if $\varsigma\in\Omega$, then the function $u_{\varepsilon
,\delta}-\varphi^{\delta}$ also reaches its (negative) minimum at $\varsigma$
since $\beta_{\varepsilon}$ is increasing. Now, due to the maximum principle,%
\begin{equation}
\mathcal{H}^{\delta}u_{\varepsilon,\delta}(\varsigma)-\mathcal{H}^{\delta
}\varphi^{\delta}(\varsigma)\geq0\geq-\gamma^{\delta}(\varsigma
)(u_{\varepsilon,\delta}(\varsigma)-\varphi^{\delta}(\varsigma)). \label{lse}%
\end{equation}
Because of (\ref{dobs}) and the assumption that $a_{0},b_{i}\in L^{\infty
}(\Omega)$ we conclude that $\mathcal{H}^{\delta}\varphi^{\delta}\geq\eta$ for
some constant $\eta$ independent of $\delta.$ Now, since $\gamma,\ f\in
L^{\infty}(\Omega)$ and using (\ref{lse}), we obtain
\begin{align*}
\beta_{\varepsilon}(u_{\varepsilon,\delta}-\varphi^{\delta})  &
=\mathcal{H}^{\delta}u_{\varepsilon,\delta}(\varsigma)+\gamma^{\delta
}(\varsigma)u_{\varepsilon,\delta}(\varsigma)-f^{\delta}(\varsigma)\\
&  \geq\mathcal{H}^{\delta}\varphi^{\delta}(\varsigma)+\gamma^{\delta
}(\varsigma)\varphi^{\delta}(\varsigma)-f^{\delta}(\varsigma)\geq c,
\end{align*}
for some constant $c$ independent of $\varepsilon$ and $\delta$ and hence
(\ref{step1}) holds. We next use (\ref{step1}) to prove that $u_{\varepsilon
,\delta}\rightarrow u$ for some function $u\in C^{2,\alpha}(\Omega)\cap
C(\overline{\Omega})$ and that $u$ is a solution to the obstacle problem
(\ref{e-obs}). To do this we first prove that there exist constants $c_{1}$
and $c_{2}$ such that%
\begin{equation}
||u_{\varepsilon,\delta}||_{L^{\infty}(\Omega)}\leq c_{2}\left(
||g||_{L^{\infty}(\Omega)}+||f||_{L^{\infty}(\Omega)}+c_{1}\right)  .
\label{u sup}%
\end{equation}
In fact, this follows by considering solutions to
\[
\left\{
\begin{array}
[c]{ll}%
\mathcal{H}^{\delta}v_{\varepsilon,\delta}-||\gamma^{\delta}||_{L^{\infty
}(\Omega)}v_{\varepsilon,\delta}=-2(||f^{\delta}||_{L^{\infty}(\Omega
)}+||\beta_{\varepsilon}(u_{\varepsilon,\delta}-\varphi^{\delta}%
)||_{L^{\infty}(\Omega)})\ \ \ \ \  & \text{in }\Omega,\\
u=||g^{\delta}||_{L^{\infty}(\Omega)} & \text{on }\partial\Omega.
\end{array}
\right.
\]
Using the maximum principle on $v_{\varepsilon,\delta}-u_{\varepsilon,\delta
},$ we see that $u_{\varepsilon,\delta}<v_{\varepsilon,\delta}.$ Moreover,
since $||\beta_{\varepsilon}(u_{\varepsilon,\delta}-\varphi^{\delta
})||_{L^{\infty}(\Omega)}$ is bounded uniformly for $\varepsilon,\delta,$ and
since the $L^{\infty}$-norm of the regularized version of a function is
bounded by the $L^{\infty}$-norm of the function itself, (\ref{u sup})
follows.\textbf{ }Then we use (\ref{step1}) and (\ref{u sup}) together with
Theorem \ref{a priori} to conclude that for every $U\subset\subset\Omega$ and
$p\geq1$ the norm $||u_{\varepsilon,\delta}||_{\mathcal{S}^{p}(U)}$ is bounded
uniformly in $\varepsilon$ and $\delta.$ Consequently $\{u_{\epsilon,\delta
}\}$ converges weakly to a function $u$ on compact subsets of $\Omega$ as
$\varepsilon,\delta\rightarrow0$ in $\mathcal{S}^{p}$, and by Theorem
\ref{Embedding} in $C^{1,\alpha}$. Also, by construction,%
\[
\underset{\varepsilon,\delta\rightarrow0}{\lim\sup}\ \beta_{\varepsilon
}(u_{\varepsilon,\delta}-\varphi^{\delta})\leq0
\]
and therefore $\mathcal{H}u+\gamma\leq f$ a.e. in $\Omega.$ In the set
$\{u\geq\varphi\}\cap\Omega$ equality holds. Together with the estimate
(\ref{step1}) this shows that $\max\{\mathcal{H}u+\gamma u-f,\varphi-u\}=0$ on
$\Omega.$ Proceeding as in the end of the proof of Theorem \ref{thm2}, using
barrier functions, we conclude that $u\in C(\overline{\Omega})$ and $u=g$ on
$\partial\Omega$, hence $u$ is a strong solution to the obstacle problem
(\ref{e-obs}). The bound (\ref{Sp bound}) is a direct consequence of the above
calculations. Altogether, this completes the proof. \hfill$\Box$

\section{Proof of Theorem \ref{Embedding}\label{SecEm}}

The embedding theorem we aim to prove is not as general as we would have
hoped, and actually, when we began working on this paper we did believe that
the proof was already out there. Despite several attempts on finding a proper
reference we were unable to find one, and in the end, we decided to add the
assumption that we are working on a homogeneous group and that the vector
fields $X_{1},...,X_{q}$ are left invariant and homogeneous of degree one
while $X_{0}$ is left invariant and homogeneous of degree two. This enables us
to prove the necessary embedding, that is that the $C^{1,\alpha}$-norm of
solutions are bounded by the $\mathcal{S}^{p}$-norm. In the case of stratified
groups this was proved by Folland in \cite[Theorem 5.15]{F75}, and no
assumption on $u$ solving a particular equation had to be made. In the pure
subelliptic content, that is, when there is no lower order term, this has been
extensively investigated, see for instance Lu \cite[Theorem 1.1]{Lu96} and the
references therein. In the subelliptic parabolic case, that is, when
$X_{0}=\partial_{t},$ this was proved in \cite[Theorem 1.4]{FGN12}. The
approach used to the case $X_{0}=\partial_{t}$ cannot be applied to this case
since we lack enough information about the fundamental solution. Finally, a
slightly less general formulation of the embedding theorem was proved in
\cite[Theorem 7]{BB00b}, where the $C^{0,\alpha}$-norm is bounded by the
$\mathcal{S}^{p}$-norm.\textbf{\ }

\noindent\textbf{Proof of Theorem \ref{Embedding}. }First, we note that by
\cite[Theorem 4]{BB00b} we have, for $\alpha=2-Q/p,$%
\[
||u||_{C^{0,\alpha}(\Omega^{\prime})}\leq c\left(  ||\mathcal{H}%
u||_{L^{r}(\Omega)}+||u||_{L^{p}}\right)  ,
\]
for some $c$ depending only on $\mathbf{G}$, $\mu$, $p$, $s$, $\Omega$ and
$\Omega^{\prime}$\ (it is stated in a slightly different way, but restricted
to our choice of $p$ and $s$ this is what is actually proved). It remains to
show that the same holds when $u$ on the left hand side is replaced by
$X_{i}u$ for $i=1,...,q.$ Let $H=\sum_{i=1}^{q}X_{i}^{2}+X_{0}$ and let
$\Gamma$ be the corresponding fundamental solution. Such a fundamental
solution exists by a classical result of Folland \cite[Theorem 2.1]{F75}.
Moreover, $\Gamma$ is homogeneous of degree $2-Q.$ This means that, for $u\in
C_{0}^{\infty}(B_{R}),$ we can write%
\[
u=Hu\ast\Gamma.
\]
Let $\phi$ be a cutoff function with $B_{R/2}(x_{0})\prec\phi\prec B_{R}%
(x_{0})$, for some $x_{0}\in\Omega,R\in\mathbf{R}$ such that $B_{2R}(x_{0}%
)\in\Omega$. That such a cutoff function exists follows from Lemma
\ref{cutoff2}. By Lemma \ref{cutoff}, $u\phi\in\mathcal{S}_{0}^{p}(B_{R}).$
Since H\"{o}lder continuity is a local property, we can restrict ourselves to
balls, and by a density argument we can look at smooth functions $u$.
Therefore, assume that $u\in C_{0}^{\infty}(\Omega)$ and let $M$ be as in
Proposition \ref{MMM}, then%
\[
X_{i}u(x)=X_{i}\ \int_{\mathbf{R}^{n}}\Gamma(y^{-1}\circ x)Hu(y)dy.
\]
Since $u$ is smooth with compact support we may differentiate inside integral,
and we obtain%
\begin{align}
|X_{i}u(x)-X_{i}u(y)|  &  \leq\int_{\mathbf{R}^{n}}|X_{i}\Gamma(z^{-1}\circ
x)-X_{i}\Gamma(z^{-1}\circ y)|\ |Hu(z)|dz\nonumber\\
&  \leq\int_{||z^{-1}\circ x||\geq M||y^{-1}\circ x||}...dz+\int%
_{||z^{-1}\circ x||<M||y^{-1}\circ x||}...dz=I+II. \label{hund}%
\end{align}
Above, it is implicitly understood that the vector fields act on $\Gamma$ as a
function of $x$ respectively $y$ (hence, do not differentiate with respect to
the $z$-variable). Since $\Gamma$ is homogeneous of degree $2-Q$ and
$X_{i},\ i=1,...,q,$ is homogeneous of degree $1,$ $X_{i}\Gamma$ is
homogeneous of degree $1-Q.$ By Proposition \ref{MMM} we get%
\[
I\leq c(\mathbf{G},p)\ ||y^{-1}\circ x||\ \int_{||z^{-1}\circ x||\geq
M||y^{-1}\circ x||}\frac{|Hu(z)|}{||z^{-1}\circ x||^{Q}}dz.
\]
Further, we introduce the sets%
\[
\sigma_{k}=\{z\in\mathbf{R}^{n}:2^{k}M||y^{-1}\circ x||\leq||z^{-1}\circ
x||\leq2^{k+1}M||y^{-1}\circ x||\},
\]
for $k=0,1,...,$ and note that the Euclidean volume of the set $\sigma_{k},$
by (\ref{bollar}), is equal to%
\begin{align}
&  |B(0,2^{k+1}M||y^{-1}\circ x||)|-|B(0,2^{k}M||y^{-1}\circ x||)|\nonumber\\
&  =|B(0,1)|\ \left(  \left(  2^{k+1}M||y^{-1}\circ x||\right)  ^{Q}-\left(
2^{k}M||y^{-1}\circ x||\right)  ^{Q}\right) \\
&  =|B(0,1)|\ (2^{Q}-1)\ 2^{Qk}M^{Q}||y^{-1}\circ x||^{Q}. \label{bolldiff}%
\end{align}
By assumption $u\in\mathcal{S}^{p}(\Omega)$ for some $p.$ Let $q$ be such that
$\frac{1}{p}+\frac{1}{q}=1.$ Then, we obtain%
\begin{align*}
I  &  \leq c(\mathbf{G},p)\ ||y^{-1}\circ x||\ \int_{||z^{-1}\circ x||\geq
M||y^{-1}\circ x||}\frac{|Hu(z)|}{||z^{-1}\circ x||^{Q}}dz\\
&  \leq c(\mathbf{G},p)\ ||y^{-1}\circ x||\ \sum_{k=0}^{\infty}\left(
2^{k}M||y^{-1}\circ x||\right)  ^{-Q}\int_{\sigma_{k}}|Hu(z)|dz\\
&  \leq c(\mathbf{G},p)\ ||y^{-1}\circ x||^{1-Q}\ \sum_{k=0}^{\infty}%
2^{-kQ}\left[  \left(  2^{Q}-1\right)  2^{kQ}M^{Q}||y^{-1}\circ x||^{Q}%
\right]  ^{1/q}\ ||Hu||_{L^{p}(\sigma_{k})}\\
&  \leq c(\mathbf{G},p)\ ||y^{-1}\circ x||^{1-Q+Q/q}\ ||Hu||_{L^{p}(\Omega
)}\sum_{k=0}^{\infty}2^{-k(Q-Q/q)}.
\end{align*}
This sum converges, and for $Q<p$ we have that the exponent of $||y^{-1}\circ
x||$ is larger than zero.

Next step is to look at $II$ in (\ref{hund}). In a similar way we define the
sets%
\[
\widetilde{\sigma}_{k}=\{z\in\mathbf{R}^{n}:2^{-(k+1)}M||y^{-1}\circ
x||\leq||z^{-1}\circ x||\leq2^{-k}M||y^{-1}\circ x||\},
\]
for $k=0,1,...,$ and in this case we get%
\[
II\leq\underset{II_{1}}{\underbrace{\int_{||z^{-1}\circ x||<M||y^{-1}\circ
x||}\frac{|Hu(z)|}{||z^{-1}\circ x||^{Q-1}}dz}}+\underset{II_{2}%
}{\underbrace{\int_{||z^{-1}\circ x||<M||y^{-1}\circ x||}\frac{|Hu(z)|}%
{||z^{-1}\circ y||^{Q-1}}dz}}.
\]
To begin with, we deal with the first term above, which by the compact support
of $u,$ is bounded by%
\begin{align*}
II_{1}  &  \leq c(\mathbf{G},p)\sum_{k=0}^{\infty}\ \int_{\widetilde{\sigma
}_{k}}\frac{|Hu(z)|}{\left(  2^{-(k+1)}M\ ||y^{-1}\circ x||\right)  ^{Q-1}%
}dz\\
&  \leq c(\mathbf{G},p)\ ||y^{-1}\circ x||^{-(Q-1)}\sum_{k=0}^{\infty
}2^{(k+1)(Q-1)}\left(  \int_{\widetilde{\sigma}_{k}}1dz\right)  ^{1/q}%
||Hu||_{L^{p}(\Omega)}\\
&  \leq c(\mathbf{G},p)\ ||y^{-1}\circ x||^{-(Q-1)}||Hu||_{L^{p}(\Omega)}%
\sum_{k=0}^{\infty}2^{(k+1)(Q-1)}\left[  \left(  2^{-k}M\ ||y^{-1}\circ
x||\right)  ^{Q}\right]  ^{1/q}\\
&  =c(\mathbf{G},p)\ ||y^{-1}\circ x||^{-(Q-1-Q/q)}||Hu||_{L^{p}(\Omega)}%
\sum_{k=0}^{\infty}2^{k(Q-1-Q/q)}.
\end{align*}
The sum converges for $Q<p,$ and in that case
\[
II_{1}\leq c(\mathbf{G},p)\ ||y^{-1}\circ x||^{(p-Q)/p}||Hu||_{L^{p}(\Omega
)}.
\]
To bound $II_{2}$, note that if $||z^{-1}\circ x||<M||y^{-1}\circ x||,$ then
$||z^{-1}\circ y||\leq c(||z^{-1}\circ x||+||y^{-1}\circ x||)\leq
c(1+M)||y^{-1}\circ x||$. This means that we can argue as for $II_{1}$, to
find that $II_{2}\leq c(\mathbf{G},p)\ ||y^{-1}\circ x||^{(p-Q)/p}%
||Hu||_{L^{p}(\Omega)}.$ Put together, we have shown that%
\[
|X_{i}u(x)-X_{i}u(y)|\leq c(\mathbf{G},p)\ ||y^{-1}\circ x||^{(p-Q)/p}%
||Hu||_{L^{p}(\Omega)}.
\]
That is, (\ref{badda}) hold for functions $u\in\mathcal{S}^{p}(\Omega)\cap
C_{0}^{\infty}(\Omega).$ The general case follows, as previously mentioned, by
using a density argument and cutoff functions. Note that, we proved this for
H\"{o}lder spaces defined by means of the distance $d_{h},$ however, this
carries over directly to our case. \hfill$\square$

\section{Homogeneous H\"{o}rmander operators\label{exempel}}

We will now give some examples as to when our results apply. The two first
examples shows operators for which our results overlap with the existing
literature, while the third and fourth example shows that our results covers
equations previously not considered for obstacle problems.

\begin{example}
(Subelliptic parabolic equations) When we replace $a_{0}X_{0}$ with
$\partial_{t}$ we get a subelliptic parabolic operator;%
\[
\mathcal{H}=\sum_{i,j=1}^{q}a_{ij}(x,t)X_{i}X_{j}+\sum_{i=1}^{q}%
b_{i}(x,t)X_{i}-\partial_{t},\ \ \ x\in\mathbf{R}^{n},t\in(0,T),n\geq3.
\]
In this case, by \cite{FGN12}, we need not assume that we have a homogeneous group.
\end{example}

\begin{example}
(Kolmogorov equations) Let%
\begin{equation}
\mathcal{H}=\sum_{i,j=1}^{q}a_{ij}(x,t)\frac{\partial^{2}}{\partial
x_{i}\partial x_{j}}+\sum_{i=1}^{q}b_{i}(x,t)\frac{\partial}{\partial x_{i}%
}+\sum_{i,j=1}^{n}c_{ij}x_{i}\frac{\partial}{\partial x_{j}}+\partial_{t},
\label{KE}%
\end{equation}
where $(x,t)\in\mathbf{R}^{n}\times\mathbf{R}$, $q<n$, with the usual
assumptions on $a_{ij}$ and $b_{i}$, while $C=\{c_{ij}\}$ is a matrix of
constant real numbers. For $(x_{0},t_{0}),$ fixed but arbitrary, we introduce
the vector fields%
\begin{equation}
X_{0}=\sum_{i,j=1}^{n}c_{ij}x_{i}\frac{\partial}{\partial x_{j}}+\partial
_{t},\ \ \ X_{i}=\frac{1}{\sqrt{2}}\sum_{j=1}^{q}a_{ij}(x_{0},t_{0}%
)\frac{\partial}{\partial x_{j}},\ \ \ i\in\{1,...,q\}. \label{XKE}%
\end{equation}
A condition which assures that $\mathcal{H}$ in (\ref{KE}) is a H\"{o}rmander
operator is that $\{X_{0},X_{1},...,X_{q}\}$ in (\ref{XKE}) satisfy the
H\"{o}rmander condition. An equivalent condition is that the matrix $C$ has
the following block structure%
\[%
\begin{pmatrix}
\ast & C_{1} & 0 & \cdots & 0\\
\ast & \ast & C_{2} & \cdots & 0\\
\vdots & \vdots & \vdots & \ddots & \vdots\\
\ast & \ast & \ast & \cdots & C_{k}\\
\ast & \ast & \ast & \cdots & \ast
\end{pmatrix}
\]
where $C_{j}$, for $j\in\{1,...,k\}$, is a $q_{j-1}\times q_{j}$ matrix of
rank $q_{j},$ $1\leq q_{k}\leq...\leq q_{1}\leq q=q_{0}.$ Further,
$q+q_{1}+...+q_{k}=n$, while $\ast$ represents arbitrary matrices with
constant entries. In the case of Kolmogorov equations, results on existence of
solutions was proved in \cite{FPP}.
\end{example}

\begin{example}
For $(x,y,z,w,t)\in\mathbf{R}^{5}$, consider the vector fields%
\[
X=\partial_{x}-xy\partial_{t},\ \ \ Y=\partial_{y}+x\partial_{w}%
,\ \ \ Z=\partial_{z}+x\partial_{t}.
\]
These vector fields satisfy H\"{o}rmanders condition since%
\[
W=[X,Y]=\partial_{w}+x\partial_{t},\ \ \ T=[X,Z]=\partial_{t}.
\]
We note that the Lie algebra generated by these vector fields are nilpotent of
step 4, but we do not have a stratified group since $\partial_{t}%
=[X,Z]=[X,W]$. Moreover, the group law $\circ$ is given by%
\[
(x,y,z,w,t)\circ(\xi,\eta,\zeta,\omega,\tau)=(x+\xi,y+\eta,z+\zeta
,w+\omega+x\eta,t+\tau-\frac{1}{2}y\xi^{2}-x\xi\eta+x\zeta+x\omega),
\]
and we can define (non-unique) translations%
\[
D_{\lambda}(x,y,z,w,t)=(\lambda x,\lambda y,\lambda^{2}z,\lambda^{2}%
w,\lambda^{3}t).
\]
This is neither a subelliptic parabolic equation, nor is it a Kolmogorov type
equation, and the results presented here is therefore new.
\end{example}

\begin{example}
(Link of groups) Following \cite{LK}, we can link groups together. The
simplest example is obtained if we define the vector fields%
\[
X_{0}=x\partial_{w}-\partial_{t},\ \ \ X_{1}=\partial_{x}+y\partial
_{s},\ \ \ X_{2}=\partial_{y}-x\partial_{s},
\]
for $(x,y,s,w,t)\in\mathbf{R}^{5}$. Then, in the variables $(x,y,s,t)$ we get
the heat operator on the Heisenberg group, while in the variables $(x,y,s,w)$
we get a Kolmogorov operator. This again defines a homogeneous H\"{o}rmander
operator, which previously have not been studied in the setting of obstacle problems.
\end{example}

\end{document}